\documentclass{article} 
\usepackage{amssymb}
\usepackage{amsmath}
\usepackage{amsthm}
\usepackage[a4paper]{geometry}
\usepackage[round]{natbib}
\usepackage[colorlinks=true, linkcolor=blue]{hyperref}

\usepackage{graphicx} 

\usepackage{amssymb}
\usepackage{amsmath,amscd}
\usepackage{amscd}
\usepackage{amsthm}
\usepackage{cancel}     %
\usepackage{centernot}  %
\usepackage{graphicx}   %
\usepackage{color}      %
\usepackage{enumitem}   %
\usepackage{commath}    %
\usepackage{algorithm}  %
\usepackage{booktabs}   %

\DeclareFontFamily{U}{mathx}{\hyphenchar\font45}
\DeclareFontShape{U}{mathx}{m}{n}{
      <5> <6> <7> <8> <9> <10> gen * mathx
      <10.95> mathx10 <12> <14.4> <17.28> <20.74> <24.88> mathx12
      }{}
\DeclareSymbolFont{mathx}{U}{mathx}{m}{n}
\DeclareMathSymbol{\intop}  {1}{mathx}{"B3}

\newcommand{\wh}{\widehat}

\let\temp\phi
\let\phi\varphi
\let\varphi\temp

\newcommand{\T}{\mathbb{T}}

\newcommand{\E}{\mathbb{E}}

\newcommand{\normalN}{\mathcal{N}}
\newcommand{\given}{\,|\,}  %

\renewcommand{\norm}[1]{\Vert#1\Vert} %

\newcommand{\eps}{\varepsilon}

\newcommand{\ges}{\gtrsim}

\newcommand{\maxop}{\vee}

\DeclareMathOperator{\tr}{tr}
\DeclareMathOperator{\supp}{supp}

\DeclareMathOperator{\cov}{cov}
\DeclareMathOperator{\var}{var}

\DeclareMathOperator*{\argmin}{arg\,min}

\DeclareMathOperator{\UniformDist}{Unif}

\DeclareMathOperator{\Rad}{Rad}

\DeclareMathOperator{\pa}{pa} %
\DeclareMathOperator{\an}{an} %
\DeclareMathOperator{\nd}{nd} %
\DeclareMathOperator{\de}{de} %
\DeclareMathOperator{\nb}{ne} %

\newcommand{\AND}{\text{ and }}

\newtheorem{theorem}{Theorem}[section]
\newtheorem{lemma}[theorem]{Lemma}
\newtheorem{corollary}[theorem]{Corollary}
\newtheorem{proposition}[theorem]{Proposition}

\newtheorem*{theorem*}{Theorem}
\newtheorem*{lemma*}{Lemma}
\newtheorem*{corollary*}{Corollary}
\newtheorem*{proposition*}{Proposition}
\newtheorem*{conjecture*}{Conjecture}

\theoremstyle{definition}

\newtheorem*{definition*}{Definition}
\theoremstyle{definition}
\newtheorem{example}{Example}
\theoremstyle{definition}
\newtheorem*{example*}{Example}
\theoremstyle{definition}

\theoremstyle{definition}
\newtheorem*{assumption*}{Assumption}
\theoremstyle{definition}
\newtheorem{condition}{Condition}
\theoremstyle{remark}
\newtheorem{remark}{Remark}
\theoremstyle{remark}
\newtheorem*{remark*}{Remark}

\renewcommand{\T}{^\top}

\newcommand{\gr}{G}

\newcommand{\truepr}{P}

\newcommand{\anc}{A}

\newcommand{\ggmcomplexity}{s\log d}

\title{Optimal estimation of Gaussian DAG models}

\usepackage{authblk}
\author[]{Ming Gao}
\author[]{Wai Ming Tai}
\author[]{Bryon Aragam}
\affil[]{\emph{University of Chicago}}

\begin{document}

\maketitle

{\let\thefootnote\relax\footnote{Contact: \texttt{\{minggao,waiming.tai,bryon\}@chicagobooth.edu}}}

\begin{abstract}%
  We study the optimal sample complexity of learning a Gaussian directed acyclic graph (DAG) from observational data. 
  Our main results establish the minimax optimal sample complexity for learning the structure of a linear Gaussian DAG model in two settings of interest: 1) Under equal variances without knowledge of the true ordering, and 2) For general linear models given knowledge of the ordering.
  In both cases the sample complexity is $n\asymp q\log(d/q)$, where $q$ is the maximum number of parents and $d$ is the number of nodes. 
  We further make comparisons with the classical problem of learning (undirected) Gaussian graphical models, showing that under the equal variance assumption, these two problems share the same optimal sample complexity. In other words, at least for Gaussian models with equal error variances, learning a directed graphical model is statistically no more difficult than learning an undirected graphical model.
  Our results also extend to more general identification assumptions as well as subgaussian errors.
\end{abstract}

\section{Introduction}\label{sec:intro}

A significant open question in the literature on structure learning is the optimal sample complexity of learning a directed acyclic graphical model. The problem of deriving upper bounds on the sample complexity for this problem goes back decades \citep{zuk2006number,friedman1996}, and in recent years there has been significant progress \citep{ghoshal2017ident,ghoshal2017sem,chen2018causal,park2017,park2018learning,park2019poisson,park2020condvar,wang2018nongauss,gao2020npvar,gao2021efficient}. Nonetheless, despite these upper bounds, a tight characterization of the optimal sample complexity is missing. This is to be contrasted with the situation for learning undirected graphs (UGs), also known as Markov random fields (MRFs), for which optimal rates were established approximately ten years ago \citep{santhanam2012information,wang2010information}, alongside similar results for support recovery in linear models \citep{wainwright2009,wainwright2009information}. In fact, this is unsurprising given the connection between these two problems via neighbourhood regression. Unfortunately, learning a directed acyclic graph (DAG) does not reduce to neighbourhood regression as it involves a more difficult order recovery step.

In this paper, we resolve this question for the special case of linear Gaussian DAG models with equal error variances. The identifiability of these models was established in \cite{peters2013}, and eventually led to the development of several polynomial-time algorithms under the equal variance assumption \citep{ghoshal2017ident,ghoshal2017sem,chen2018causal,gao2020npvar}. Nonetheless, it was not known whether or not any of these algorithms were optimal for this precise statistical setting.
We will show that a variant of the EQVAR algorithm 
from \cite{chen2018causal} is indeed optimal. This involves the derivation of new lower bounds
and a novel analysis of the EQVAR algorithm that sharpens the existing sample complexity upper bound from $O(q^{2}\log d)$ to $O(q\log(d/q))$, where $q$ is the maximum number of parents in the DAG and $d$ is the number of nodes. This upper bound is optimal up to constants, and allows for the high-dimensional regime with $d\gg n$, where as usual $n$ denotes the sample size. 
Moreover, in Section~\ref{sec:lb:given_ordering}, we extend this result to the case of general Gaussian models with known ordering.
Our results also extend to more general identification assumptions (e.g. allowing for unequal error variances) as well as subgaussian error terms; see Remark~\ref{rem:gen}.

As a problem of independent interest, we further compare the complexity of learning Gaussian graphical models (GGMs) and Gaussian DAG models under the equal variance assumption. Given the additional complexity of the order recovery problem in DAG learning, the folklore has generally been that learning DAGs is harder than learning UGs. Despite this folklore, few results are available to rigorously characterize the hardness of these problems on an equal footing (besides known NP-hardness results for both problems, see \citealp{srebro2003maximum,chickering1996,chickering2004}). The equal variance assumption gives us the opportunity to make an apples-to-apples comparison under the same assumptions. As we will show, the optimal sample complexity for \emph{both problems} scales as $O(q\log(d/q))$. In other words, learning a DAG is \emph{statistically} no harder than learning a GGM under the equal variance assumption.
It is worth emphasizing that this comparison is purely statistical: The \emph{computational} complexity of the algorithm we analyze is exponential in $q$ whereas learning GGMs can be done efficiently; see also Remark~\ref{rem:lasso}.

To the best of our knowledge, these are the first results giving a tight characterization of the optimal sample complexity for learning DAG models from observational data. 

The rest of this paper is organized as follows: In the remainder of Section~\ref{sec:intro}, we discuss related work and the problem setting. In Sections~\ref{sec:ub} and~\ref{sec:lb} we present our main results for learning equal variance DAGs. Then in Section~\ref{sec:lb:given_ordering} we consider the special case of known ordering, and in Section~\ref{sec:ug} make further comparisons with learning undirected GGMs. An illustrative simulation study is presented in Section~\ref{sec:expt} before concluding with some open questions in Section~\ref{sec:conc}.

\paragraph{Notation and preliminaries}
Given a directed graph $G=(V,E)$ with $|V|=d$ nodes, we make the following standard definitions: 
\begin{itemize}
\item The parents $\pa(k)=\{\ell:(\ell,k)\in E\}$;
\item The descendants $\de(k)$ to which $k$ has at least one directed path;
\item The nondescendents $\nd(k) = V\setminus \de(k)$;
\item The ancestors $\an(k)$ any of which has at least one directed path to $k$.
\end{itemize}
When $(\ell,k)\in E$ we will often write $\ell\to k$ for short.
A source node is any $k\in V$ with $\an(k)=\emptyset$. A subgraph $G[V\setminus S]$ is the original graph with nodes in $S$ and edges related to $S$ removed. Given a DAG $G$, the moralized graph $\mathcal{M}(G)$ is constructed by dropping the orientations of all directed edges and then connecting all nodes within $\pa(k)$ for all $k\in [d]$. 
A \emph{topological sort} (also called an \emph{ordering}) of a DAG $G$ is an ordering of the nodes $V$ such that $j\to k\implies j\prec k$.

Given a random vector $X=(X_{1},\ldots, X_{d})$, we say that $G$ is a Bayesian network for $X$ (or more precisely, its joint distribution $P$), if the following factorization holds:
\begin{align}
\label{eq:defn:bn}
    P(X)
    = \prod_{k=1}^{d}P(X_{k}\given \pa(k)).
\end{align}
In this case, we abuse notation by identifying the random vector $X$ with the vertex set $V$, i.e. $V=X=[d]=\{1,2,\ldots,d\}$.
We denote the class of all DAGs with $d$ nodes and at most $q$ parents per node (i.e. in-degree $\le q$) by $\mathcal{\gr}_{d,q}$. 

\subsection{Related work}
To provide context, we begin by reviewing the related problem of learning the structure of an undirected graph (e.g. MRF, GGM, etc.) from data. Early work establishing consistency and rates of convergence includes \citet{meinshausen2006,banerjee2008,ravikumar2010}, with information-theoretic lower bounds following in \citet{santhanam2012information,wang2010information}. More recently, sample optimal and computationally efficient algorithms have been proposed \citep{vuffray2016interaction,misra2020information}. Part of the reason for the early success of MRFs is owed to the identifiability and convexity of the underlying problems. By contrast, DAG learning is notably \emph{nonidentifiable} and \emph{nonconvex}. This has led to a line of work to better understand identifiability \citep[e.g.][]{hoyer2009,zhang2009,peters2014,peters2013,park2017} as well as efficient algorithms that circumvent the nonconvexity of the score-based problem \citep{ghoshal2017ident,ghoshal2017sem,chen2018causal,gao2020npvar,gao2021efficient}. The latter class of algorithms begins by finding a topological sort of the DAG; once this is known the problem reduces to a variable selection problem. Our paper builds upon this line of work. 

Other approaches include score-based learning, for which various consistency results are known \citep{geer2013,buhlmann2014,loh2014causal,aragam2015highdimdag,nowzohour2016,nandy2018,rothenhausler2018causal,aragam2019globally}, but for which optimality results are missing. It is interesting to note that recent work has explicitly connected the equal variance assumption we use here to score-based learning via a greedy search algorithm \citep{rajendran2021structure}.
We also note here important early work on the constraint-based PC algorithm, which also establishes finite-sample rates under the strong faithfulness assumption \citep{kalisch2007}.

For completeness, we pause for a more detailed comparison with existing sample complexity upper bounds from the literature. \citet{geer2013} studied the $\ell_{0}$-penalized MLE and showed that $n/\log n\ges d$ samples suffice, which was later improved to $n\ges s\log d$ \citep{aragam2019globally}. 
Using a different approach, \citet{ghoshal2017ident} proved that $n\ges s^4\log d$ samples suffice, where $s$ is the maximum Markov blanket size or equivalently the size of the largest neighbourhood in the conditional independence graph of $X$.
The dependency on $s$ arises from the way this algorithm uses the inverse covariance matrix $\Gamma=\Sigma^{-1}$. 
Moreover, their result additionally requires the restricted strong adjacency faithfulness assumption, which we do not impose. 
In a more recent work, \citet{chen2018causal} show that $n\ges q^{2}\log d$ samples suffices to learn the ordering of the underlying DAG, but do not establish results for learning the full DAG. Similar to our work, \citet{chen2018causal} do not make any faithfulness or restricted faithfulness-type assumptions. We note also the work of \citet{park2020condvar} that establishes rates of convergence assuming $n>d$, but which precludes the high-dimensional scenario $d\gg n$. For comparison, we improve these existing bounds to $n\ges q\log(d/q)$ for the full DAG and moreover prove a matching lower bound (up to constants). 
\citet{ghoshal2017limits} have also established lower bounds for a range of DAG learning problems up to Markov equivalence. For example, their lower bound for sparse Gaussian DAGs is $\sigma^{2}(q\log(d/q) + q^{2}/d)/(\sigma^{2}+2w_{\max}^{2}(1+w_{\max}^{2}))$, where $w_{\max}$ depends on the $\ell_{2}$ norms of the regression coefficients. By contrast, our lower bounds depend instead on $(\beta_{\min},M)$ (cf. \eqref{eq:defn:betamin}, \eqref{eq:defn:M} for definitions).

\subsection{Problem setting}
Although our results extend to more general settings, we focus on the special case of linear Gaussian Bayesian networks under equal variances. See Remark~\ref{rem:gen} for a discussion of generalizations. Specifically, let $X=(X_{1},\ldots, X_{d})$,
\[
X_k = \langle \beta_k, X \rangle + \epsilon_k, \ \ \ \ \var(\epsilon_k) \equiv \sigma^2,  \ \ \ \ \E[\epsilon_k]=0,\
\]
that is, each node is a linear combination of its parents with independent Gaussian noise. The variance $\sigma^{2}$ of each noise term is assumed to be the same; this is the key identifiability assumption that is imposed on the model. 

More compactly, let $B=(\beta_{jk})$ denote the coefficient matrix  such that $\beta_{jk}\ne 0$ is equivalent to the existence of the edge $j \to k$. Then letting $\epsilon=(\epsilon_1,\ldots,\epsilon_d)$ we have
\begin{align}
\label{eq:lin:model}
X = B\T X + \epsilon.
\end{align}
The matrix $B$ defines a graph $G=G(B)$ by its nonzero entries, i.e.
\begin{align*}
G(B) 
= (V, E(B)),
\quad
\left\{
\begin{aligned}
V 
&= X, \\
E(B)
&= \{(j,k) : \beta_{jk}\ne 0\}.
\end{aligned}
\right.
\end{align*}
Whenever $G$ is acyclic, it is easy to check that \eqref{eq:defn:bn} holds, and hence $G$ is a Bayesian network for $X$. In the sequel we assume that $G$ is acyclic.

The following quantities are important in the sequel:
The largest in-degree of any node is denoted by $q$, i.e.
\begin{align}
\label{eq:defn:q}
q
= q(B)
:= \sup_{k}|\pa(k)|
= \sup_{k}|\supp(\beta_{k})|.
\end{align}
The absolute values of the coefficients are lower bounded by $\beta_{\min}$, i.e. 
\begin{align}
\label{eq:defn:betamin}
\beta_{\min}  
= \beta_{\min}(B)
:= \min\{|\beta_{jk}| : \beta_{jk}\ne 0\}.
\end{align}
Furthermore, assume the covariance matrix $\Sigma = \E [XX\T]$ satisfies
\begin{align}
\label{eq:defn:M}
M^{-1} \le \lambda_{\min}(\Sigma) \le \lambda_{\max}(\Sigma) \le M
\end{align}
for some $M>1$. 

Let the class of distributions satisfying the above conditions \eqref{eq:defn:q}, \eqref{eq:defn:betamin}, and \eqref{eq:defn:M} be denoted by $\mathcal{F}_{d,q}(\beta_{\min},\sigma^2,M)$. For any $F\in\mathcal{F}_{d,q}(\beta_{\min},\sigma^2,M)$ we have
\begin{align}
\label{eq:cov}
\Sigma 
= \sigma^{2}(I-B)^{-T}(I-B)^{-1}.
\end{align}
This follows directly from \eqref{eq:lin:model} and $\cov(\eps)=\sigma^{2}I$.
Since the DAG is identifiable from the observational distribution, we denote $\gr(F)$ to be the DAG associated with the distribution $F\in \mathcal{F}_{d,q}(\beta_{\min},\sigma^2,M)$. 
Finally, we introduce the variance gap:
\begin{align*}
    \Delta \equiv \min_k\min_{\substack{\anc\subseteq\nd(k) \\  \pa(k) \setminus \anc \ne \emptyset \\ \anc\subseteq \nd(\pa(k) \setminus \anc)}}\E_\anc\var(X_k\given \anc) - \sigma^2 
\end{align*}
where the subscript indicates that the expectation is being taken over the random variables in $A$.
This is the missing conditional variance on ancestors if not all the parents are conditioned on, which serves as the identifiability signal for the main algorithm. It turns out it can be explicitly expressed in terms of the edge coefficient and noise variance:
\begin{lemma}\label{lem:Delta}
$\Delta = \beta^2_{\min}\sigma^2 > 0$.
\end{lemma}
The proof of this lemma is a straightforward calculation; see Appendix~\ref{app:Delta} for details.
\begin{remark}
\label{rem:gen}
Both our upper and lower bounds can be generalized as follows: Although we assume Gaussianity for simplicity, everything extends to subgaussian families without modification. This is because the upper bound analysis relies only on subgaussian concentration, and the lower bounds easily extend to subgaussian models (i.e. since subgaussian also contains Gaussian as a subclass).
Furthermore, the equal variance assumption can be relaxed to more general settings as long as $B$ can be identified by Algorithm~\ref{alg:eqvar}. Examples include (a) the ``unequal variance'' condition from \citet{ghoshal2017sem} (see Assumption~1 therein) and (b) if noise variances are known up to some ratio as in \citet{loh2014causal}. Moreover, both of these identifiability conditions include the naive equal variance condition as a special case, hence the lower bounds still apply. This implies more general optimality results for a wider class of Bayesian networks. 
\end{remark}

\section{Algorithm and upper bound}\label{sec:ub}

We begin with stating the sufficient conditions on the sample size for DAG recovery under the equal variance assumption.
Namely, we present an algorithm (Algorithm \ref{alg:eqvar}) that takes samples from a distribution $F\in \mathcal{F}_{d,q}(\beta_{\min},\sigma^2,M)$ as an input and returns the DAG $G(F)$ with high probability.
We first
 state an upper bound for the number of samples required in Algorithm \ref{alg:eqvar} in Theorem \ref{thm:ub}.

\begin{theorem}\label{thm:ub}
For any $F\in\mathcal{F}_{d,q}(\beta_{\min},\sigma^2,M)$, let $\wh{\gr}$ be the DAG return by Algorithm~\ref{alg:eqvar} with $\gamma=\Delta/2$. If
\[
n \gtrsim \frac{M^5}{\Delta}\bigg(q\log \frac{d}{q} + \log \delta \bigg) \, ,
\]
then $\truepr(\wh{\gr} = \gr(F)) \gtrsim 1 - \delta$.
\end{theorem}
The proof of this result can be found in Appendix~\ref{app:ub}. The obtained sample complexity depends on the variance gap $\Delta$, which serves as signal strength, and covariance matrix norm $M$, which shows up when estimating conditional variances. Treating these parameters as fixed, the sample complexity scales with $q\log(d/q)$. The order of this complexity arises mainly from counting all possible conditioning sets. The proof follows the correctness of Algorithm \ref{alg:eqvar}, which consists of two main steps: \textit{Learning ordering} and \textit{Learning parents}.

Algorithmically, the first step is the same as \citet{chen2018causal}, however, our analysis is sharper: We separately analyze the estimation of each conditional variance directly rather than indirectly via the inverse covariance matrix.
This leads to the improved sample complexity in Theorem~\ref{thm:ub}. This step is where we exploit the equal variance assumption: The conditional variance $\var(X_k\given C)$ of each random variable $X_k$ is a constant $\sigma^2$ if and only if $\pa(k)\subseteq C$ for any nondescendant set $C$.
This implies that the variance of any non-source node in the corresponding subgraph would be larger than $\sigma^2$.
Therefore, when all conditional variances $v_{kC}$ are correctly estimated with error within some small factor of the signal $\Delta$ (see Lemma~\ref{lem:ub:est}), identifying the node with the smallest $\sigma_k$ yields a source node in the underlying subgraph. Recall that $\sigma_k$ is the minimum variance estimation that node $k$ can achieve conditioned on at most $q$ nondescendants. Finally, recursively applying the above step leads to a valid topological sort.

In the second step, given the correct ordering, we use Best Subset Selection (BSS) along with a backward phase to learn the parents for each node. 
Note that BSS is already applied in the step \textbf{1.(c).i.} of Algorithm~\ref{alg:eqvar} and the candidate set $C_j$ can be stored for each $\wh{\tau}_j$, thus there is no additional computational cost. 
Again, when all conditional variances are well approximated by their sample counterpart $v_{kC}$, $C_j$ would be a superset of the true parents of current node $\wh{\tau}_j$, otherwise the minimum would not be achieved. 
Meanwhile, removal of any true parent $i\in\pa(\wh{\tau}_j)$ from $C_j$ would induce a significant change in conditional variances, which is quantified by $\Delta$ as well. This is used to design a tuning parameter $\gamma$ in the backward phase for pruning $C_j$.
Finally, we show the tail probability of conditional variance estimation error is well bounded to get the desired sample complexity in Lemma~\ref{lem:ub:tailprob}.

\begin{algorithm}[t]
\caption{\textsc{LearnDAG} algorithm}
\label{alg:eqvar}
\textbf{Input:} Sample covariance matrix $\wh{\Sigma}=\frac{1}{n}\sum_{i=1}^n X_iX_i\T$, backward phase threshold $\gamma$\\
\textbf{Output:} $\widehat{\gr}$.

\begin{enumerate}
    \item \textit{Learning Ordering}:
        \begin{enumerate}
            \item Initialize empty ordering $\wh{\tau}=[]$
            \item Denote $v_{kC}: = \wh{\Sigma}_{kk} - \wh{\Sigma}_{kC}\wh{\Sigma}_{CC}^{-1}\wh{\Sigma}_{Ck}$
            \item For $j=1,2,\ldots, d$
            \begin{enumerate}
                \item Calculate $\sigma_{k}: = \min_{C\subseteq  \wh{\tau}, |C|\le q} v_{kC}$
                \item Update $\wh{\tau}=[\wh{\tau}, \argmin_k \sigma_{k}]$
            \end{enumerate}
        \end{enumerate}

    \item \textit{Learning Parents}:
        \begin{enumerate}
        \item Initialize empty graph $\wh{\gr}=\emptyset$
        \item For $j=1,2,\ldots,d$
            \begin{enumerate}
                \item Let $C_j = \argmin_{C\subseteq  \wh{\tau}_{[1:j-1]}, |C| \le q} v_{\wh{\tau}_jC}$
                \item Set
                \begin{align*}
                    \hspace*{-1cm}
                    \pa_{\wh{\gr}}(\wh{\tau}_j)  = C_j \setminus \bigg\{i \in C_j \bigg| |v_{\wh{\tau}_j C_j}- v_{\wh{\tau}_j C_j\setminus i}|\le \gamma \bigg\}
                \end{align*}
            \end{enumerate}
        \end{enumerate}
    \item \textit{Return} $\widehat{\gr}$
\end{enumerate}
\end{algorithm}

When the true variance gap $\Delta$ is unknown, we can select the tuning parameter $\gamma$ according to the following theorem:
\begin{theorem}\label{thm:ub:tuning}
For any $F\in\mathcal{F}_{d,q}(\beta_{\min},\sigma^2,M)$, let $\wh{\gr}$ be the DAG return by Algorithm~\ref{alg:eqvar} with tuning parameter
\[
\gamma \asymp \frac{2M^5 q\log(d/q)}{n}\,.
\]
If 
\[
n \gtrsim \frac{M^5}{\Delta}q\log \frac{d}{q}\, ,
\]
then $\truepr(\wh{\gr} = \gr(F)) \gtrsim 1 - \exp(-q\log(d/q))$.
\end{theorem}
The proof of this result can be found in Appendix~\ref{app:ub:tuning}. 
\begin{remark}
\label{rem:lasso}
A computationally attractive alternative to BSS is the Lasso, or $\ell_{1}$-regularized least squares regression. Unlike BSS, the Lasso requires restrictive incoherence-type conditions. If these conditions (or related conditions such as irrepresentability) are imposed on each parent set, then the Lasso can be used to recover the full DAG under a similar sample complexity scaling \citep[see e.g.][]{wainwright2009}. Furthermore, these incoherence-type conditions can be further relaxed through the use of nonconvex regularizers such as the MCP \citep{zhang2010} or SCAD \citep{fan2001}; see also \citet{loh2014nonconvex}.
\end{remark}

\section{Lower bound}\label{sec:lb}

We will now present the necessary conditions on the sample size for DAG recovery under the equal variance assumption.
Namely, we present a subclass of  $\mathcal{F}_{d,q}(\beta_{\min},\sigma^2,M)$ such that any estimator that successfully recovers the underlying DAG in this subclass with high probability requires a prescribed minimum sample size.
For this, we rely on Fano's inequality, which is a standard technique for establishing necessary conditions for graph recovery.
See Corollary~\ref{col:lb:fano} for the exact variant we use.

\begin{theorem}\label{thm:lb}
Assume $q\le d/2$. If 
\[
n \le (1-2\delta)\max\bigg( \frac{\log d}{\beta_{\min}^2}, \frac{q\log(d/q)}{M^2-1} \bigg)
\]
then for any estimator $\wh{\gr}$,
\[
\sup_{F\in \mathcal{F}_{d,q}(\beta_{\min},\sigma^2,M)}\truepr(\wh{\gr}\ne G(F)) \ge \delta - \frac{\log 2}{\log d} \, .
\]
\end{theorem}

In Theorem \ref{thm:lb}, we state two sample complexity lower bounds: $q\log(d/q)/(M^2-1)$ and $\log d/\beta_{\min}^2$. Though the first one dominates when fixing other parameters as constants, the second one reveals the dependency on the signal strength $\beta_{\min}$. This can also be seen from the upper bound in Theorem~\ref{thm:ub} by replacing $\Delta=\beta_{\min}^2\sigma^2$ (cf. Lemma~\ref{lem:Delta}).
We will present two ensembles for each bound.
The first one is the whole set of sparse DAGs $\mathcal{G}_{d,q}$, and the second is the set of DAGs with only one edge, which is constructed to study the dependency on the coefficient $\beta_{\min}$.

For the first ensemble, we borrow the ideas from \citet{santhanam2012information} to count the number of DAGs in $\mathcal{G}_{d,q}$.
The only difference is we consider DAGs instead of undirected graphs.
Also, it is easy to bound the KL divergence between any two distributions in this ensemble due to Gaussianity, which would lead to the bound $q\log(d/q)/(M^2-1)$.
For the second ensemble, it is easy to count the size of this ensemble since we consider the DAGs with only one edge. Then all possibilities of any different pair of edges are analyzed to bound the KL divergence.
This ensemble gives us the bound $\log d/\beta_{\min}^2$. The detailed proof can be found in Appendix~\ref{app:lb}.

For comparison, \citet{ghoshal2017limits} previously established a lower bound for general Gaussian DAGs (i.e. without equal variances) of 
\begin{align*}
\Omega\Big(\sigma^{2}\frac{q\log(d/q) + q^{2}/d}{\sigma^{2}+2w_{\max}^{2}(1+w_{\max}^{2})}\Big),
\end{align*} 
where $w_{\max}$ depends on the $\ell_{2}$ norms of the regression coefficients. 
Holding $\sigma^2$ constant, $w_{\max}^2$ is similar to the maximum marginal variance of the variables, which is comparable with our definition of $M$ as an upper bound on $\norm{\Sigma}$.
By contrast, under the stronger assumption of equal variances, our lower bound is 
\begin{align*}
\Omega\Big(\frac{\log d}{\beta_{\min}^2}\maxop \frac{q\log(d/q)}{M^2-1} \Big),
\end{align*}
which is
a comparable lower bound. This is interesting since by restricting to simpler equal variance models (i.e. a smaller family), the problem should become easier, however, our analysis shows this is not the case.
In particular, our lower bounds do not follow from previous work, and require a slightly different analysis as outlined in Appendix~\ref{app:lb}.

\section{Reconstructing a DAG from its ordering}\label{sec:lb:given_ordering}

The second step of Algorithm~\ref{alg:eqvar} may be of interest in its own right: Abstracted away, this step seeks to reconstruct a DAG from knowledge of its topological sort.
We claim that the second step of Algorithm~\ref{alg:eqvar} is in fact sample optimal for learning the parents of each node (and hence all of $G$) given the true ordering of $\gr$ under more general assumptions. 

Dropping the equal variance condition from $\mathcal{F}_{d,q}(\beta_{\min},\sigma^2,M)$, define $\sigma^2_k:=\var(\epsilon_k)$ and let $\overline{\mathcal{F}}_{d,q}(\beta_{\min},\sigma_{\max}^2,M)$ denote the class of Gaussian distributions such that \eqref{eq:defn:q}, \eqref{eq:defn:betamin}, and \eqref{eq:defn:M} hold and
\begin{align*}
\sup_{k}\sigma^2_k\le \sigma^2_{\max},
\end{align*}
i.e. $\sigma^2_k$ is allowed to depend on $k$.
Note that
$\mathcal{F}_{d,q}(\beta_{\min},\sigma_{\max}^2,M)\subset \overline{\mathcal{F}}_{d,q}(\beta_{\min},\sigma_{\max}^2,M)$. 
Furthermore, we modify the definition of the variance gap for $\overline{\mathcal{F}}_{d,q}(\beta_{\min},\sigma_{\max}^2,M)$ as follows:
\begin{align*}
    \overline{\Delta} \equiv \min_k\min_{\substack{\anc\subseteq\nd(k) \\  \pa(k) \setminus \anc \ne \emptyset}}\E_\anc\var(X_k\given \anc) - \sigma_k^2 \,.
\end{align*}
Finally, given a known topological sort $\tau$ of $G$, let $\wh{\gr}(\tau)$ be the DAG returned by the second step of Algorithm~\ref{alg:eqvar} with $\gamma=\overline{\Delta}/2$.

\begin{remark}
As with the rest of our results, these result extend to subgaussian models without issue. See Remark~\ref{rem:gen}.
\end{remark}

Using the second part of Lemma~\ref{lem:ub:est}, Lemma~\ref{lem:ub:tailprob} and following the proof in Appendix~\ref{app:ub:proof}, we have an upper bound on the sample complexity for recovering $G$ from its ordering:
\begin{proposition}\label{prop:ub:given_ordering}
For any $F\in\overline{\mathcal{F}}_{d,q}(\beta_{\min},\sigma_{\max}^2,M)$, given a valid topological sort $\tau$ of $\gr$, let $\wh{\gr}(\tau)$ be the DAG returned by the second step of Algorithm~\ref{alg:eqvar} with $\gamma=\overline{\Delta}/2$. If
\[
n \gtrsim \frac{M^5}{\overline{\Delta}}\bigg(q\log \frac{d}{q} + \log \delta \bigg) \, ,
\]
then $\truepr(\wh{\gr}(\tau) = \gr(F)\given \tau) \gtrsim 1 - \delta$.
\end{proposition}
The ``given $\tau$'' in the probability is to emphasize that the estimator has the access to the true ordering $\tau$. 

Unsurprisingly, this approach of using best subset selection with a backwards phase is indeed optimal: We have a matching lower bound (up to constants).
\begin{proposition}\label{prop:lb:pa_given_ordering}
If 
\[
n \le \frac{\sigma_{\max}^2}{8M\beta_{\min}^2}q\log \frac{d}{q}\,,
\]
then given the knowledge of true ordering $\tau$ of DAG $\gr$, for any estimator $\widehat{\gr}$,
\[
\sup_{F\in \overline{\mathcal{F}}_{d,q}(\beta_{\min},\sigma_{\max}^2,M)}\truepr(\wh{\gr}\ne G(F) \given \tau) \ge \frac{1}{2} \, .
\]
\end{proposition}
The proof uses known lower bounds from the sparse support recovery literature \citep{wainwright2009information}; see Appendix~\ref{app:given_ordering} for details. 

This more general optimality result for the second step shows that it is only in the first step (learning parents) that the equal variance assumption is operational. Moreover, although it may be possible to improve the sample complexity of the second step for the smaller class $\mathcal{F}_{d,q}(\beta_{\min},\sigma^2,M)$, since the sample complexity upper bound of the first step of Algorithm~\ref{alg:eqvar} matches the lower bound for recovering the whole graph, such improvements would not change the optimal sample complexity for learning $G$.

\section{Comparison with undirected graphs}
\label{sec:ug}

Our results on DAG learning under equal variances raise an interesting question: \emph{Is learning an equal variance DAG statistically more difficult than learning its corresponding Gaussian graphical model (i.e. inverse covariance matrix)?} This is especially intriguing given the folklore intuition that learning a DAG is more difficult than learning an undirected graph (UG). 
In fact, it is common to learn an undirected graph \emph{first} as a pre-processing step in order to reduce the search space and sample complexity for DAG learning \citep{perrier2008,loh2014causal,buhlmann2014,aragam2019globally}. In this section we explore this question and show that in fact, at least in the special case of equal variance Gaussian models, the sample complexity of both problems is the same.

\subsection{Gaussian graphical models}

First, let us recall some basics about undirected graphical models, also known as Markov random fields (MRFs). When $X\sim\normalN(0,\Sigma)$ as in this paper, an MRF can be read off from the inverse covariance matrix $\Gamma=(\gamma_{jk}):=\Sigma^{-1}$. More precisely, the zero pattern of $\Gamma$ defines an undirected graph $U=U(\Gamma)$ that is automatically an MRF for $X$:
\begin{align*}
U(\Gamma) 
= (V, E(\Gamma)),
\quad
\left\{
\begin{aligned}
V 
&= X, \\
E(\Gamma)
&= \{(j,k) : \gamma_{jk}\ne 0\}.
\end{aligned}
\right.
\end{align*}
Let $\nb(k)=\{\ell\in V : (\ell,k)\in E\}$ be the neighbours of node $k$, i.e. they are connected by some edge. Note the distinction between the parents of $k$ in a directed graph vs. the neighbours of $k$ in an undirected graph. This model is often referred as the Gaussian graphical model (GGM). 

\citet{wang2010information} showed that the optimal sample complexity for learning a GGM is $n\asymp\ggmcomplexity$,
where 
\begin{align}
\label{eq:def:s}
s:=\max_{k}|\nb(k)|
\end{align}
is the degree of $U$ or maximum neighborhood size, and \citet{misra2020information} developed an efficient algorithm that matches this information-theoretic lower bound. 
Given $F\in\mathcal{F}_{d,q}(\beta_{\min},\sigma^2,M)$, let $U(F)$ be the undirected graph induced by the covariance matrix of $F$ (cf. \ref{eq:cov}).
It follows that for any $F\in\mathcal{F}_{d,q}(\beta_{\min},\sigma^2,M)$ we can learn the structure $U(F)$ with $\Theta(\ggmcomplexity)$ samples. Note that this sample complexity scales with $s$ instead of $q$. 

\begin{example}
\label{ex:ug:gap}
Consider the DAGs $G_{1}$ and $G_{2}$ in Figure~\ref{fig:comp}. In $G_{1}$, we have $q=O(d)$ since $T$ has $d$ parents, whereas in $G_{2}$ we have $q=O(1)$ since each $S_{k}$ has only one parent. Thus, we expect that learning $G_{1}$ will require $\Theta(d)$ samples and learning $G_{2}$ will require $\Theta(\log d)$ samples. By comparison, the UGs associated with each model, given by $U_{1}$ and $U_{2}$ have $s=d$, and hence if we use the previous approaches to learn each $U_{k}$ we will need $\Theta(d\log d)$ samples each. Of course, this is to be expected: One should expect that a specialized estimator that exploits the structure of the family $\mathcal{F}_{d,q}(\beta_{\min},\sigma^2,M)$ to perform better.
\end{example}

\begin{figure}[t]
    \centering
    \includegraphics[width=0.7\linewidth]{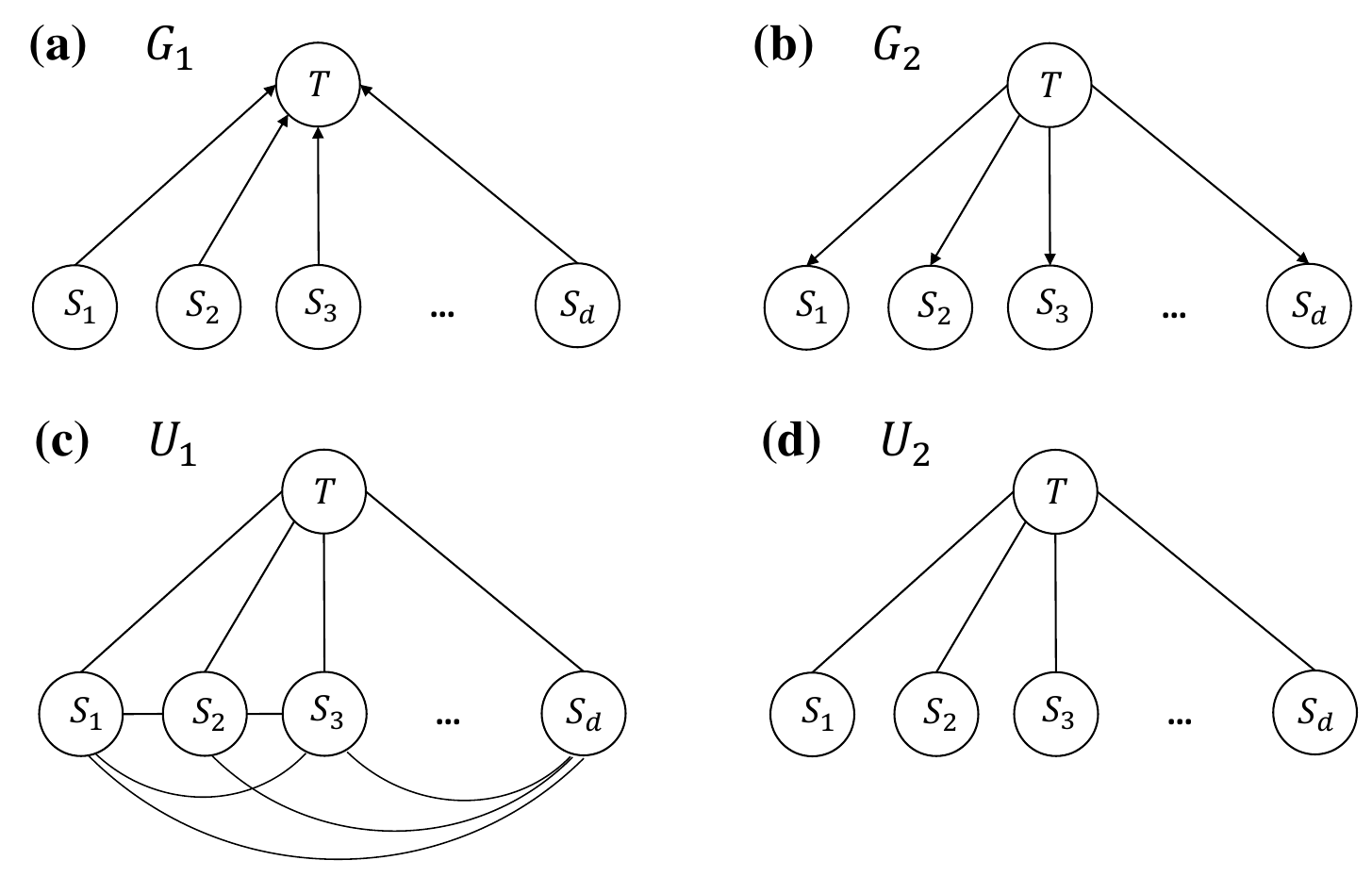}
    \caption{Illustrating examples to compare sample complexities for learning DAG and UG. \textbf{(a) \& (b):} The original DAGs; \textbf{(c) \& (d):} The UGs generated under equal variance condition by DAGs in (a) \& (b) respectively.}
    \label{fig:comp}
\end{figure}

\subsection{Optimal estimation of equal variance GGMs}

Example~\ref{ex:ug:gap} shows that there is a gap between existing ``universal'' algorithms for learning GGMs (i.e. algorithms that do not exploit the equal variance assumption) and the sample complexity for learning equal variance DAGs. A natural question then is: \emph{What is the optimal sample complexity for learning the structure of $U(F)$ for any $F\in\mathcal{F}_{d,q}(\beta_{\min},\sigma^2,M)$?}

We begin by establishing a lower bound that matches the lower bound in Theorem~\ref{thm:lb} (up to constants):
\begin{theorem}\label{thm:ug:lb}
Assume $q\le d/2$. If 
\begin{align*}
    n \le \max\bigg(\frac{2(1-\delta)\log d}{\beta_{\min}^2},
 \frac{(1-2\delta)q\log(d/q)}{M^2-1} \bigg) \, ,
\end{align*}
then for any estimator $\wh{U}$,
\[
\sup_{F\in \mathcal{F}_{d,q}(\beta_{\min},\sigma^2,M)}\truepr(\wh{U}\ne U(F)) \ge \delta - \frac{1}{\log d} \, .
\]
\end{theorem}
The proof is deferred to Appendix~\ref{app:comp}. 

To derive an upper bound for this problem, we use the well-known trick of moralization; see \citet{lauritzen1996} for details. Since $F\in\mathcal{F}_{d,q}(\beta_{\min},\sigma^2,M)$, we can first learn the DAG $G=G(F)$ via Algorithm~\ref{alg:eqvar}. Given the output $\wh{G}$, we then form the moralized graph and define $\wh{U}:=\mathcal{M}(\wh{G})$. 
See Algorithm~\ref{alg:ug}. 

\begin{algorithm}[t]
\caption{\textsc{LearnUG} algorithm}
\label{alg:ug}
\textbf{Input:} Sample covariance matrix $\wh{\Sigma}=\frac{1}{n}\sum_{i=1}^n X_iX_i\T$, backward phase threshold $\gamma$\\
\textbf{Output:} $\widehat{U}$.
\begin{enumerate}
    \item \textit{Learn DAG}: Let
    $\wh{G} = \textsc{LearnDAG}(\wh{\Sigma},\gamma)$
    \item \textit{Moralization}:
    Set $\wh{U} = \mathcal{M}(\wh{G})$
    \item \textit{Return} $\wh{U}$
\end{enumerate}
\end{algorithm}

This approach is justified by a result due to \citet{loh2014causal}. First, we need the following condition:
\begin{condition}\label{cond:ugmoral}
Let precision matrix $\Gamma=\Sigma^{-1} = [\cov(F)]^{-1}$, $\Gamma_{ij} = 0 \Leftrightarrow \beta_{ij}=0 \AND \beta_{ik}\beta_{jk}=0$ for all $k\ne i,j$.
\end{condition}
When we sample nonzero entries of $B$ from some continuous distribution independently, Condition~\ref{cond:ugmoral} is satisfied except on a set of Lebesgue measure zero. For example, it is easy to check that the examples in Figure~\ref{fig:comp} satisfy this condition. Under this condition, moralization is guaranteed to return $U$:
\begin{lemma}[Theorem~2, \citealp{loh2014causal}]\label{lem:lohbum}
If Condition~\ref{cond:ugmoral} holds, then $U = \mathcal{M}(\gr(F))$.
\end{lemma}

Under Condition~\ref{cond:ugmoral}, we have the following upper bound, which matches the lower bound in Theorem~\ref{thm:ug:lb}:
\begin{theorem}\label{thm:ug:ub}
Assuming Condition~\ref{cond:ugmoral}, for any $F\in\mathcal{F}_{d,q}(\beta_{\min},\sigma^2,M)$ let $\wh{U}$ be the UG returned by Algorithm~\ref{alg:ug} with $\gamma=\Delta/2$. If 
\[
n \gtrsim \frac{M^5}{\Delta}\bigg(q\log \frac{d}{q} + \log \delta \bigg) \, ,
\]
then $\truepr(\wh{U} = U(F)) \gtrsim 1 - \delta$.
\end{theorem}
The proof is straightforward by Theorem~\ref{thm:ub} and Lemma~\ref{lem:lohbum}. This answers the question proposed at the beginning of this section: \emph{Under the equal variance assumption, learning a DAG is no harder than learning its corresponding UG.}

\begin{remark}
\label{rem:open}
It is an interesting question whether or not similar results hold without Condition~\ref{cond:ugmoral}; i.e. is there a direct estimator of $U$---not based on moralizing a DAG---that matches the sample complexity of learning an equal variance DAG?
\end{remark}

\begin{remark}
To compare Theorem~\ref{thm:ug:lb} with previous work,
\citet{wang2010information} showed the optimal sample complexity is $\Theta(s\log d / \lambda^2)$ where $\lambda^2=\min_{(s,t)\in E}\Gamma_{st}^2/(\Gamma_{ss}\Gamma_{tt})$, which is dominated by $\beta_{\min}^2$ when $\beta_{\min}$ is small.
For equal variance GGMs, our result gives $\Omega(q\log(d/q) / M^2 + \log d / \beta^2_{\min})$.
Under Condition~\ref{cond:ugmoral}, $s$ is always greater than $q$, so our new lower bound is strictly smaller, along with a matching upper bound that shows the $s$ dependence for general GGMs is suboptimal for equal variance GGMs.
Another related work \citep{cai2016estimating} derives lower bounds 
on precision matrix estimation 
under certain matrix norms, which is distinct from the graph recovery problem we consider in this work.
For comparison, put in our setting, their lower bound becomes $\Omega(s^2\log d)$, which again depends on $s$ instead of $q$.
\end{remark}

\section{Experiments}\label{sec:expt}

To illustrate the effectiveness of Algorithm~\ref{alg:eqvar}, we report the results of a simulation study. 
We note that existing variants of Algorithm~\ref{alg:eqvar} have been compared against other approaches such as greedy DAG search (GDS, \citealp{peters2013}), see \citet{chen2018causal} for details.
In our experiments, we controlled the number of parents when generating the DAG, fix all noise variances to be the same, and sample nonzero entries of $\beta_k$'s uniformly from given intervals.

\begin{figure*}[t]
    \centering
    \includegraphics[width=1.\linewidth]{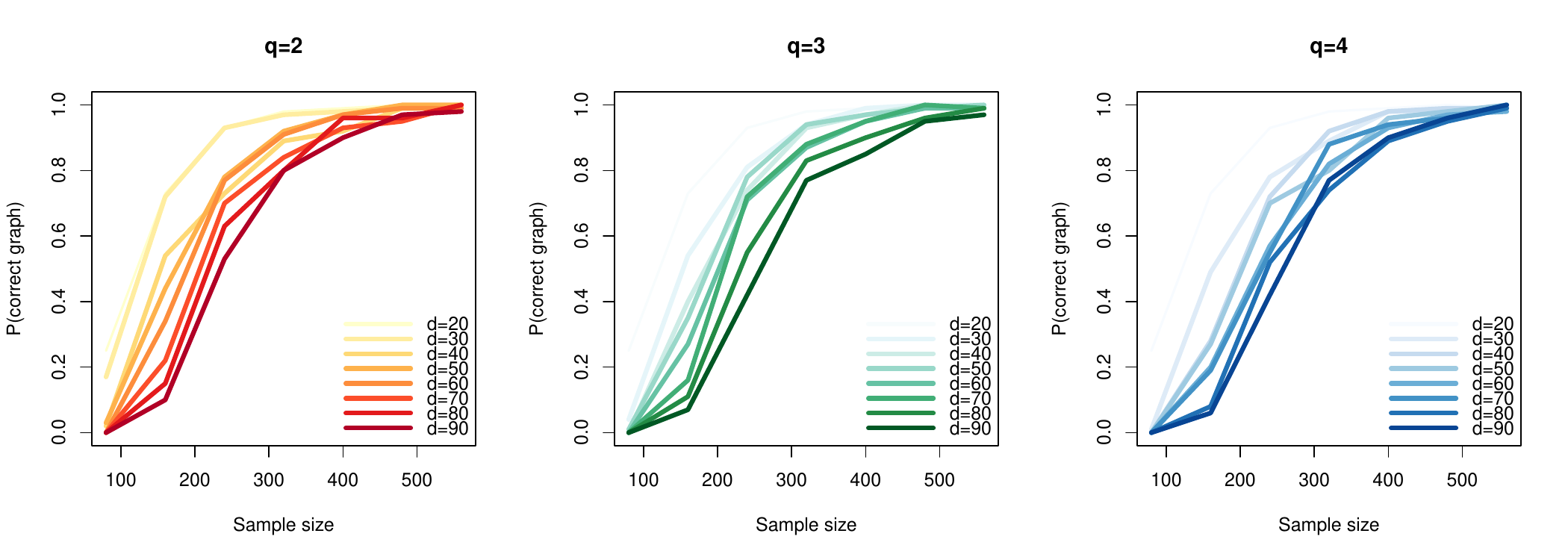}
    \caption{$\truepr(\wh{\gr}=\gr)$ v.s. sample size $n$. Different in-degree $q$ are presented by different colors. The darker shade indicates large number of nodes $d$.}
    \label{fig:main}
\end{figure*}

\subsection{Experiment settings}\label{app:expt:sett}
To generate random DAGs, we first randomly permute $[d]$ to obtain an ordering $\tau$. Then for each $j\in[d]$, we randomly draw a set $S$ of $q$ nodes from $\tau_{[1:j-1]}$ and set 
\begin{align*}
\pa_G(\tau_j) = \begin{cases}
\tau_{[1:j-1]}, & j\le q+1 \\
S, & \text{otherwise}.
\end{cases}
\end{align*}

We then generate random nonzero coefficients according to $\beta_j\sim \Rad\times \UniformDist(0.5,1)$, where $\Rad$ is a Rademacher random variable. Finally, we generate data by the resulting Gaussian linear model:
\[
X_k = \beta_k\T X + \epsilon_k, \ \ \ \ \epsilon_k\sim\mathcal{N}(0,\sigma^2)\ \  \forall k
\]
with $\sigma=0.3$.

We consider graphs with $d \in \{20,30,\ldots,90\}$ nodes and in-degree $q\in\{2,3,4\}$. For each setting, the total number of replications is $N=100$. For each replication, we generate a random graph and a dataset with sample size $n\in\{80,160,\ldots,560\}$. Finally, we report $\#\{\wh{\gr}=\gr\} / N$ to approximate $\truepr(\wh{\gr}=\gr)$.

\subsection{Implementation}\label{app:expt:impl}
We implement the \textit{learning ordering} phase of Algorithm~\ref{alg:eqvar} using code from \citet{chen2018causal}\footnote{The code can be found at \url{https://github.com/WY-Chen/EqVarDAG/blob/master/R/EqVarDAG_HD_TD.R}.}, which inputs the oracle in-degree $q$ and outputs a topological sort. For the \textit{learning parents} phase of Algorihm~\ref{alg:eqvar}, we use the R package \texttt{leaps} \citep{lumley2013package} for Best Subset Selection with BIC. 

The experiments were conducted on an internal cluster using an Intel E5-2680v4 2.4GHz CPU with 64 GB memory.

\subsection{Results}
The results are shown in Figure~\ref{fig:main}. As expected, the probability of successfully recovering true DAG goes to one quickly across different settings of the number of nodes $d$ and maximum in-degree $q$. Since Best Subset Selection is computationally expensive, we are not able to examine higher dimensions systematically. Nonetheless, to test higher dimensional cases, we checked several (random) cases for $d = 200,300,500$, and the results shows with 90\% chance the DAG is successfully recovered for moderate sample sizes $n=480,560,900$.

\section{Conclusion}\label{sec:conc}

In this paper, we derived the optimal sample complexity for learning Gaussian linear DAG models under an equal variance condition that has been extensively studied in the literature. These results extend to subgaussian errors under similar assumptions as well as more general models with unequal variances as long as the DAG remains identifiable by the proposed algorithm, which is easy to implement and simulations corroborate our theoretical findings. 
We also investigated the sub-problem of learning a linear DAG from its ordering and made comparisons with the classical problem of learning GGMs, showing the sample complexity of both problems is the same.

We conclude with some open questions.
Although our algorithm is sample optimal, it is not computationally efficient. As noted in Remark~\ref{rem:lasso}, an easy fix is to use $\ell_{1}$-regularization, however, this would require imposing restrictive incoherence assumptions. It would be interesting to find efficient algorithms without such conditions, along the lines of \citet{misra2020information} for GGMs. We conjecture that there is such a polynomial-time algorithm that achieves the optimal sample complexity bound without such restrictive conditions.
It also is not known whether or not score-based approaches \citep[e.g.][]{geer2013,loh2014causal,nandy2018,aragam2019globally,rajendran2021structure} are sample optimal.
Another interesting question posed in Remark~\ref{rem:open} is whether or not there is a moralization-free algorithm that achieves the optimal sample complexity for learning $U(F)$ established in Theorem~\ref{thm:ug:lb}. This would allow Condition~\ref{cond:ugmoral} to be relaxed or removed entirely.

Finally, it would be of interest to generalize the results in Section~\ref{sec:ug} to more general families, i.e. beyond equal variances and its generalizations \citep{ghoshal2017sem}. This would require the derivation of new identifiability conditions, as in \citet{gao2021efficient} and \citet{rajendran2021structure}.

\bibliography{gsopmdag-foo}
\bibliographystyle{abbrvnat}

\appendix

\section{Proof of upper bound}\label{app:ub}

\subsection{Preliminaries}\label{app:ub:prel}
We first show that if all conditional variances are estimated sufficiently well, then Algorithm~\ref{alg:eqvar} is able to identify the true DAG.
\begin{lemma}\label{lem:ub:est}
If for all $k\in V$ and $C\subset V\setminus \{k\}$, $|C|\le q$,
\[
|\var(X_k\given X_C) - v_{kC}|\le \Delta/4
\]
then $\wh{\gr}=\gr$.
\end{lemma}
\begin{proof}
We start by showing that $\wh{\tau}$ is a valid ordering for $\gr$, which is equivalent to saying $\wh{\tau}_j$ is a source node of the subgraph $\gr[V\setminus \wh{\tau}_{[1:j-1]}]$ for all $j$. We proceed by induction. For $j=1$, it reduces to compare marginal variances. 
\[
\begin{cases}
\var(X_k) = \sigma^2 & k \text{ is a source node} \\
\var(X_\ell) > \sigma^2 + \Delta & \ell \text{ is not a source node} 
\end{cases}
\]
For any non-source node $\ell$ and any source node $k$,
\begin{align*}
    \sigma_\ell &= v_{\ell \emptyset} \\
    &\ge \var(X_\ell) - \Delta/4 \\
    & > \Delta \times \frac{3}{4} + \sigma^2 \\
    & > \Delta /4  + \sigma^2 \\
    & =  \var(X_k) + \Delta/4 \\
    & \ge v_{k\emptyset} = \sigma_k
\end{align*}
Thus $k$ is preferred over $\ell$. Given that $\wh{\tau}_{[1:j-1]}$ correctly identified, by the equal variance assumption,
\begin{align*}
    & \min_{C\subseteq \wh{\tau}_{[1:j-1]}, |C|\le q}\var(X_k\given C) = \var(X_k\given \pa(k)) = \sigma^2  \\
    & \ \ \ \ \text{ if } k \text{ is a source node of } \gr[V\setminus \wh{\tau}_{[1:j-1]}] \\
    & \min_{C\subseteq \wh{\tau}_{[1:j-1]}, |C|\le q}\var(X_\ell\given C) > \Delta +  \sigma^2  \\
    & \ \ \ \  \text{ if }  \ell \text{ is not a source node of } \gr[V\setminus \wh{\tau}_{[1:j-1]}] 
\end{align*} 
Therefore, for any $\ell$ that is not a source node and for any $k$ that is a source node,
\begin{align*}
    \sigma_\ell &= \min_{C\subseteq \wh{\tau}_{[1:j-1]}, |C|\le q} v_{\ell C} \\
    &\ge \min_{C\subseteq \wh{\tau}_{[1:j-1]}, |C|\le q}\var(X_\ell\given C) - \Delta/4 \\
    & > \frac{3}{4}\Delta  + \sigma^2 \\
    & > \Delta /4  + \sigma^2 \\
    & = \min_{C\subseteq \wh{\tau}_{[1:j-1]}, |C|\le q} \var(X_k\given X_{C}) + \Delta/4 \\
    & \ge \min_{C\subseteq \wh{\tau}_{[1:j-1]}, |C|\le q} v_{kC} = \sigma_k
\end{align*}
Thus the first step of Algorithm~\ref{alg:eqvar} will always include $k$ instead of $\ell$ into $\wh{\tau}$. This implies that $\wh{\tau}$ is a valid topological ordering.

Now we look at the second step of Algorithm~\ref{alg:eqvar}, this step is to remove false parents from candidate set returned by Best Subset Selection. For any $j$, let $\wh{\tau}_j=j$ for ease of notation. Given that $\wh{\tau}$ is a valid ordering, $\pa(j)\subseteq \wh{\tau}_{[1:j-1]}$. We first conclude $\pa(j)\subseteq C_j$, otherwise there exists $C_j'\subseteq \wh{\tau}_{[1:j-1]}$ with $\pa(j)\subseteq C_j'$ such that
\begin{align*}
    v_{jC_j} & \ge \var(X_j\given C_j) - \Delta/4 \\
    & > \sigma^2 + \Delta - \Delta/4 \\
    & > \sigma^2 + \Delta/4 \\
    & = \var(X_j\given C_j') + \Delta/4\\
    & \ge v_{jC_j'}
\end{align*}
Then $C_j$ should not lead to minimum. Then for any $k\in \pa(j)$ and any $\ell\in C_j\setminus \pa(j)$,
\begin{align*}
    v_{jC_j} - v_{jC_j\setminus k} & \ge \var(X_j\given C_j) - \var(X_j\given C_j\setminus k) - \Delta /2  \\
    & > \Delta/2 = \gamma \\
    v_{jC_j} - v_{jC_j\setminus \ell} & \le \var(X_j\given C_j) - \var(X_j\given C_j\setminus \ell) + \Delta /2  \\
    & = \Delta/2 = \gamma
\end{align*}
Thus $\ell$ will be removed while $k$ will stay. Then $\pa_{\wh{\gr}}(j) = \pa_\gr(j)$ for all $j$.
\end{proof}
Next we bound the estimation error tail probability:
\begin{lemma}\label{lem:ub:tailprob}
For all $k\in V$ and $C\subset V\setminus \{k\}$, $|C|\le q$,
\[
\truepr(|\var(X_k\given X_C) - v_{kC}|\ge \epsilon)  \le A_1\exp(-A_2n\epsilon/ M^5 + q)
\]
for some constants $A_1,A_2$.
\end{lemma}
\begin{proof}
Denote the covariance between $k$ and set of nodes $C$ at
\[
\Sigma_t=
\begin{pmatrix}
\Sigma_{kk} & \Sigma_{kC} \\
\Sigma_{Ck} & \Sigma_{CC}
\end{pmatrix}
\]
Note that $\rVert\Sigma_{kC} \lVert, \rVert \Sigma_{CC}\lVert,\rVert \Sigma^{-1}_{CC} \lVert \le \rVert\Sigma_t \lVert \le \lVert\Sigma\rVert \le M$.

Then the estimation error for conditional variance
\begin{align*}
    &\big|\var(X_k\given X_C) - \widehat{\var}(X_k\given X_C) \big| \\
    & = \big|(\Sigma_{kk} - \Sigma_{kC} \Sigma_{CC}^{-1}\Sigma_{Ck} ) - ( \widehat{\Sigma}_{kk} - \widehat{\Sigma}_{kC} \widehat{\Sigma}_{CC}^{-1}\widehat{\Sigma}_{Ck}) \big| \\
    & \le \big|\widehat{\Sigma}_{kk} - \Sigma_{kk}\big| + \big|\widehat{\Sigma}_{kC}(\widehat{\Sigma}_{CC}^{-1} - \Sigma_{CC}^{-1})\widehat{\Sigma}_{Ck}\big| \\ &\qquad+\big|(\widehat{\Sigma}_{kC}-\Sigma_{kC})\Sigma_{CC}^{-1}\widehat{\Sigma}_{Ck}\big| +\big|\Sigma_{kC}\Sigma^{-1}_{CC}(\widehat{\Sigma}_{Ck} - \Sigma_{Ck})\big|\\
    & \le \big|\widehat{\Sigma}_{kk} - \Sigma_{kk}\big| + \lVert\widehat{\Sigma}_{CC}^{-1} - \Sigma_{CC}^{-1}\rVert \lVert\widehat{\Sigma}_{kC}\rVert^2 \\ 
    &\qquad+ M\lVert\widehat{\Sigma}_{kC}-\Sigma_{kC}\rVert  \lVert\widehat{\Sigma}_{Ck}\rVert + M^2 \lVert\widehat{\Sigma}_{Ck} - \Sigma_{Ck}\rVert \\
    & \le \big|\widehat{\Sigma}_{kk} - \Sigma_{kk}\big| \\ 
    &\qquad+ 2\lVert\widehat{\Sigma}_{CC}^{-1} - \Sigma_{CC}^{-1}\rVert \lVert\widehat{\Sigma}_{kC} - \Sigma_{kC}\rVert^2 + 2M^2\lVert\widehat{\Sigma}_{CC}^{-1} - \Sigma_{CC}^{-1}\rVert \\ 
    &\qquad+ M\lVert\widehat{\Sigma}_{kC}-\Sigma_{kC}\rVert^2 + M\lVert\widehat{\Sigma}_{kC}-\Sigma_{kC}\rVert \\
    & +M^2 \lVert\widehat{\Sigma}_{Ck} - \Sigma_{Ck}\rVert \\
    & \le \big|\widehat{\Sigma}_{kk} - \Sigma_{kk}\big| \\ 
    &\qquad+ 2\lVert\widehat{\Sigma}_{CC}^{-1} - \Sigma_{CC}^{-1}\rVert  \lVert\widehat{\Sigma}_t - \Sigma_{t}\rVert^2 + 2M^2\lVert\widehat{\Sigma}_{CC}^{-1} - \Sigma_{CC}^{-1}\rVert \\ 
    &\qquad+ M\lVert\widehat{\Sigma}_{t}-\Sigma_{t}\rVert^2 + M\lVert\widehat{\Sigma}_{t}-\Sigma_{t}\rVert +M^2 \lVert\widehat{\Sigma}_{t} - \Sigma_{t}\rVert
\end{align*}
The first inequality is by the triangular inequality, and the second simply bounds $\lVert\Sigma_{kC}\rVert$ by $M$. The third inequality introduces the estimation error of $\widehat{\Sigma}_{kC}$ and the final inequality replaces this with the estimation error of the full covariance matrix $\widehat{\Sigma}_t$. To set the RHS to be smaller than $\epsilon>0$, we consider three estimation errors. The first two can be controlled via standard sub-exponential concentration, whereas the third can be controlled via Theorem 6.5 from \citet{wainwright2019high}:
\begin{align*}
    & \truepr(|\widehat{\Sigma}_{kk} - \Sigma_{kk}| \ge \zeta) \le \exp(-A_3n\zeta / M) \\
    & \truepr(\lVert\widehat{\Sigma}_{t} - \Sigma_{t}\rVert \ge \zeta) \le \exp(-A_4n\zeta / M + q)
\end{align*}
for some constants $A_3,A_4$. The largest error is from 
\begin{align*}
    \lVert\widehat{\Sigma}_{CC}^{-1} - \Sigma_{CC}^{-1}\rVert & \le \lVert\Sigma^{-1}_{CC}\rVert\lVert\widehat{\Sigma}_{CC}^{-1}\rVert\lVert\Sigma_{CC} - \widehat{\Sigma}_{CC}\rVert \\
    &\le M\lVert\widehat{\Sigma}_{CC}^{-1}\rVert\lVert\Sigma_{CC} - \widehat{\Sigma}_{CC}\rVert \\
    & = M (\lVert\widehat{\Sigma}^{-1}_{CC}\rVert -\lVert \Sigma_{CC}^{-1}\rVert) \lVert\Sigma_{CC} - \widehat{\Sigma}_{CC}\rVert \\
    &\qquad+ M \lVert\Sigma_{CC}^{-1}\rVert\lVert\Sigma_{CC} - \widehat{\Sigma}_{CC}\rVert \\
    & \le M  \lVert\widehat{\Sigma}_{CC}^{-1} - \Sigma_{CC}^{-1}\rVert\lVert\Sigma_{CC} - \widehat{\Sigma}_{CC}\rVert \\
    & + M^2 \lVert\Sigma_{CC} - \widehat{\Sigma}_{CC}\rVert
\end{align*}
After some arrangement, we have
\begin{align*}
     \lVert\widehat{\Sigma}_{CC}^{-1} - \Sigma_{CC}^{-1}\rVert \le \frac{M^2\lVert\Sigma_{CC} - \widehat{\Sigma}_{CC}\rVert}{1- M \lVert\Sigma_{CC} - \widehat{\Sigma}_{CC}\rVert} \le \zeta
\end{align*}
as long as 
\[
\lVert\Sigma_{CC} - \widehat{\Sigma}_{CC}\rVert \le \frac{\zeta}{M^2 + M\zeta} \le \frac{\zeta}{M^2}\,.
\]
This is just another Gaussian covariance matrix estimation error, i.e.
\[
\truepr(\lVert\widehat{\Sigma}_{CC} - \Sigma_{CC}\rVert > \zeta/M^2) \le \exp(-A_5n\zeta / M^3 + q)\,,
\]
for some constant $A_5$. Now we require all the errors to be bounded by $\zeta = \epsilon/M^2$ such that the conditional variance estimation error is within $\epsilon$. Thus
\begin{align*}
    &\truepr(|\var(X_k|X_C) -\widehat{\var}(X_k|X_C)| > \epsilon)\\
    &\le \exp(-A_3n\epsilon/M^3) + \exp(-A_4n\epsilon/M^3 + q) \\
    &\qquad+ \exp(-A_5n\epsilon/M^5 + q) \\
    & \le A_1\exp(-A_2n\epsilon/M^5 + q)\,.\qedhere
\end{align*}
\end{proof}

\subsection{Proof of Theorem~\ref{thm:ub}}\label{app:ub:proof}
\begin{proof}
Combine Lemma~\ref{lem:ub:est} and \ref{lem:ub:tailprob}, we can have success probability:
\begin{align*}
    \truepr(\wh{\gr}\ne\gr) & \le \truepr\bigg(\bigcup_{\substack{k\in V \\ C\subset V\setminus\{k\} \\ |C|\le q}} \{|\var(X_k\given X_C) - v_{kC}| \ge \epsilon \}\bigg) \\
    &\le \sum_{\substack{k \in V \\ C\subset V\setminus\{k\}  \\ |C|\le q}}\truepr\bigg( |\var(X_k\given X_C) - v_{kC}| \ge \epsilon \bigg)\\
    & \lesssim d\times (\binom{d-1}{1} + \ldots + \binom{d-1}{q})\times \exp(q - n\epsilon / M^5) \\
    & \lesssim d\times q \times e^{q\log(d/q)}\times \exp(q - n\epsilon / M^5) \\
    & = \exp(q\log(d/q) + \log d + \log q + q - n\epsilon / M^5) \\
    & \asymp \exp(q\log(d/q) -n\epsilon/M^5)\,.
\end{align*}
The last equality is by $(\frac{d}{q})^q = (\frac{d}{q})^{q-1}\times (\frac{d}{q}) \geq 2^{q-1}\times (\frac{d}{q}) \geq q\times  \frac{d}{q} = d$, thus $q\log(d/q) \gtrsim \log d$. In the end, replace $\epsilon$ by $\Delta/4$, solve for the sample size $n$ such that
\[
\exp(q\log(d/q) - n\Delta/M^5) \lesssim \delta\,,
\]
we can have the desired sample complexity.
\end{proof}

\subsection{Proof of Theorem~\ref{thm:ub:tuning}}\label{app:ub:tuning}
\begin{proof}
In the proof of Lemma~\ref{lem:ub:est}, denote the estimation error to be upper bounded by $\epsilon$, i.e. for all $k\in V$ and $C\subset V\setminus \{k\}$, $|C|\le q$,
\[
|\var(X_k\given X_C) - v_{kC}|\le \epsilon \,,
\]
then it suffices to have
\[
2\epsilon < \gamma < \Delta/2 - 2\epsilon
\]
for the correctness of second phase to proceed. Therefore, let $\gamma = 3\epsilon$ and require $\epsilon < \Delta/10$. Finally, set
\[
\epsilon \asymp \frac{2M^5\times q\log (d/q)}{n}
\]
Then we have failure probability bounded:
\begin{align*}
    \truepr(\wh{\gr}\ne\gr) & \lesssim\exp(q\log(d/q) - n\epsilon/M^5)\\
    &= \exp(-q\log(d/q)) \,.
\end{align*}
And to satisfy the requirement $\epsilon< \Delta/10$, we need sample size
\begin{align*}
n&\gtrsim \frac{M^5 q\log(d/q)}{\Delta}\,.\qedhere
\end{align*}
\end{proof}

\section{Proof of lower bound}\label{app:lb}
\subsection{Preliminaries}\label{app:lb:prel}
Let's start with recalling Fano's inequality and its corollary under the structure learning setting. Let $\theta(F)$ be a  parameter associated to some observational distribution $F$.
\begin{lemma}[\citealp{yu1997assouad}, Lemma 3]\label{lem:lb:fano}
For a class of distributions $\mathcal{F}$ and its subclass $\mathcal{F}' = \{F_1,\ldots,F_N\}\subseteq \mathcal{F}$, 
\[
\inf_{\wh{\gr}}\sup_{F\in\mathcal{F}} \E_F \mathbf{dist}(\theta(F), \wh{\theta}) \ge \frac{s}{2}\bigg(1 - \frac{n\alpha + \log 2 }{\log N}\bigg)
\]
where 
\begin{align*}
\alpha &= \max_{F_j\ne F_k\in\mathcal{F}'}\mathbf{KL}(F_j|| F_k) \\ 
s& =\max_{F_j\ne F_k\in\mathcal{F}'}\mathbf{dist}(\theta(F_j),\theta(F_k))    
\end{align*}
\end{lemma}
Set $\theta(F)=\gr(F)$, $\mathbf{dist}(\cdot,\cdot) = \mathbf{1}\{\cdot \ne \cdot\}$. One consequence of Lemma~\ref{lem:lb:fano} is as follows:
\begin{corollary}\label{col:lb:fano}
Consider some subclass $\mathcal{\gr}'=(G_1,\ldots,G_N)\subseteq \mathcal{\gr}_{d,q}$, and let $\mathcal{F}' = \{F_1,\ldots,F_N\}$, each of whose elements is generated by one distinct $\gr \in \mathcal{G}'$. If the sample size is bounded as
\[
n \le \frac{(1-2\delta)\log N}{\alpha} \,,
\]
then the any estimator for $\gr$ is $\delta$-unreliable:
\[
\inf_{\wh{\gr}}\sup_{F\in\mathcal{F}_{d,q}(\beta_{\min},\sigma^2,M)} \truepr(\wh{\gr}\ne \gr(F)) \ge \delta - \frac{\log 2}{\log N}
\]
\end{corollary}
Thus the \textbf{strategy} for building lower bound is to find a subclass of original problem such that
\begin{itemize}
    \item Has large cardinality $N$;
    \item Pairwise KL divergence between any two distributions is small.
\end{itemize}
Now we do some counting for the number of DAGs with $d$ nodes and in-degree bounded by $q$.
\begin{lemma}\label{lem:lb:numdag}
For $q\le d/2$,
the number of DAGs with $d$ nodes and in-degree bounded by $q$ scales as $\Theta(dq\log(d/q))$.
\end{lemma}
\begin{proof}
The proof construction is similar to \citet{santhanam2012information}. We can upper bound by number of directed graphs (DG), and lower bound by one particular subclass of DAGs.

For upper bound, note that a DG has at most $d^2$ many, and a sparse DAG has at most $dq$ any edges. Since $q \le d/2$, then we have $dq \le d^2 / 2$. Since for $\ell<k\le d^2/2$, we have $\binom{d^2}{\ell} < \binom{d^2}{k}$, and there are $\binom{d^2}{k}$ DGs with exactly $k$ directed edges, then the numbder of DGs is upper bounded by 
\[
\log|DGs| = \log\sum_{k=1}^{dq} \binom{d^2}{k} \le \log dq\binom{d^2}{dq} \asymp dq\log\frac{d}{q}
\]

For lower bound, we look at one subclass of DAGs. Suppose $d/(q+1)$ is an integer, otherwise discard remaining nodes. First partition $d$ nodes into $q+1$ groups with equal size $d/(q+1)$.
Then for the first group, build directed edges from nodes in group $2,3,\ldots,q+1$ to group $1$, which requires $q$ permutations on $d/(q+1)$ nodes within one particular group. Then the nodes in group $1$ has exactly degree $q$. Similarly, for group $2$, build directed edges from group $3,4,\ldots,q+1$, which requires $q-1$ permutations on $d/(q+1)$ nodes. Therefore, for the subclass of DAGs generated in this way of partition, we have
\[
(\frac{d}{q+1}!)^{(q) + (q-1) +\ldots + 1} = (\frac{d}{q+1}!)^{q(q+1)/2}
\]
many DAGs, any of which is valid DAG and has degree bounded by $q$. Then the cardinality
\[
\log |DAGs| \ge \log (\frac{d}{q+1}!)^{q(q+1)/2} \asymp  dq\log \frac{d}{q}
\]
Thus the total number of DAGs scales as $\Theta(dq\log\frac{d}{q})$.
\end{proof}
\subsection{Proof of Theorem~\ref{thm:lb}}\label{app:lb:proof}
\begin{proof}
We consider two ensembles:
\paragraph{Ensemble A}
In this ensemble, we consider all possible DAGs with bounded in-degree. Note that by Lemma~\ref{lem:lb:numdag}, we know $N\asymp dq\log(d/q)$, it remains to provide an upper bound for the KL divergence between any two distributions within the class. For any two $F_j,F_k\in \mathcal{F}_{d,q}(\beta_{\min},\sigma^2,M)$, denote their covariance matrices to be $\Sigma_j,\Sigma_k$. Due to Gaussianity, It is easy to see that
\begin{align*}
    \mathbf{KL}(F_i|| F_j)  &=\frac{1}{2}\bigg(\E_{F_j}[X\T\Sigma_k^{-1}X] - \E_{F_j}[X\T\Sigma_j^{-1}X] \bigg) \\
    &=\frac{1}{2}\bigg(\E_{F_j}[X\T\Sigma_k^{-1}X] - d\bigg) \\
    &= (\tr(\Sigma^{-1}_k\Sigma_j) - d) / 2\\
    &\le (\sqrt{\tr(\Sigma_j^{-2})\tr(\Sigma_k^2)} - d)/2 \\
    & \le (M^2-1)d \,.
\end{align*}
Therefore, we can establish the first lower bound that
\[
n \gtrsim \frac{q\log(d/q)}{M^2-1}
\]
\paragraph{Ensemble B}
For this ensemble, we consider the DAGs with exactly one edge $u\to v$ and coefficient $\beta_{\min}$, denoted as $\gr^{uv}$. There are 2 directions and $d(d-1)/2$ many edges, so the cardinality of this ensemble would be $N=d(d-1)\asymp d^2$. Then denote the distribution defined according to $G^{uv}$ as $F^{uv}$, the log likelihood
\begin{align*}
-\log F^{uv} &\propto \frac{1}{2\sigma^2} \bigg[\sum_{i\notin \{u,v\}}X_i^2 + X_u^2 + (X_v-\beta_{\min} X_u)^2\bigg] \\
& + d\log \sigma
\end{align*}
and the difference between any two cases is
\begin{align*}
    \log F^{uv}  - \log F^{jk} &= \frac{1}{2\sigma^2}\bigg[ X_v^2 + (X_k-\beta_{\min} X_j)^2 \\
    &- X_k^2 - (X_v-\beta_{\min} X_u)^2 \bigg]
\end{align*}
Then take expectation over $F^{uv}$ we get the KL divergence:
\[
\mathbf{KL}(F^{uv}|| F^{jk})= \frac{1}{2\sigma^2}\bigg[\beta_{\min}^2\sigma^2 + \E_{F^{uv}}(\beta_{\min}^2X_j^2 - 2X_kX_j )\bigg]
\]
For fixed edge $(u,v)$, any other edges $(j,k)$ has relationship and corresponding KL divergence below:
\begin{itemize}
    \item $j\ne u, k\ne v$, $\mathbf{KL}(F^{uv}|| F^{jk})=\beta_{\min}^2$
    \item $j=u,k\ne v$, $\mathbf{KL}(F^{uv}|| F^{jk})=\beta_{\min}^2$
    \item $j\ne u,k=v$, $\mathbf{KL}(F^{uv}|| F^{jk})=\beta_{\min}^2$
    \item $ j=v, k=u$, $\mathbf{KL}(F^{uv}|| F^{jk})=\beta_{\min}^2 + \beta_{\min}^4/2 - \beta_{\min}$
    \item $j=v, k\ne u$, $\mathbf{KL}(F^{uv}|| F^{jk})=\beta_{\min}^2 + \beta_{\min}^4/2$
    \item $j\ne v,k=u$, $\mathbf{KL}(F^{uv}|| F^{jk})=\beta_{\min}^2$
\end{itemize}
Among them, the largest KL between $F^{uv},F^{jk}$ is $\beta_{\min}^2+\beta_{\min}^4/2$. Therefore, we can conclude a lower bound 
\begin{align*}
n 
&\gtrsim \frac{1}{\beta_{\min}^4+2\beta_{\min}^2} \log d \asymp \frac{\log d}{\beta_{\min}^2}.
\qedhere
\end{align*}
\end{proof}

\section{Proof of Proposition~\ref{prop:lb:pa_given_ordering}}\label{app:given_ordering}
\begin{proof}
For simplicity we consider DAGs with $d+1$ nodes. We first recall a known lower bound for sparsity recovery: Consider the linear model $Y=\beta \T X + \epsilon$ with $X\sim \mathcal{N}(0,\Sigma)$ and $\epsilon\sim\mathcal{N}(0,\sigma^2)$. The support of $\beta$ is $S\subset [d]$ and $|S|=q$. Let $\beta_{\min} := \min_{j:\beta_j\ne 0}|\beta_j|$. Then informally, 
\begin{lemma}[\citet{wainwright2009information}, Theorem~2]\label{lem:given_ordering:lb}
If
\[
n \le \frac{\log \binom{d}{q}}{8\omega_{bu}(\Sigma)\frac{\beta_{\min}^2}{\sigma^2}}
\]
where 
\[
    \omega_{bu}(\Sigma):=\E_S \bigg[\min_{|z_j|\ge 1\, \forall j}z_S\T \Sigma_{SS}z_S\bigg]\,,
\]
then with $q$ known, for any instance from the linear model and any estimator $\widehat{S}$ for $S$,
\[
\truepr(\widehat{S}\ne S) \ge \frac{1}{2}\,.
\]
\end{lemma}
Now we adapt this result to our setting, $\omega_{bu}(\Sigma)\beta_{\min}^2$ is the variance explained by $X$ under regression model, thus upper bounded by $M\beta^2_{\min}$ when regarding $X$ as parents in DAG. Additionally, since every Gaussian with positive definite $\Sigma$ has a minimal I-map and given ordering, the parents can be read off through regression, i.e. the model class in Lemma~\ref{lem:given_ordering:lb} is equivalent to the one generated in $\overline{\mathcal{F}}_{d,q}(\beta_{\min},\sigma^2_{\max},M)$. 
For any estimator $\widehat{\gr}=\widehat{\gr}(\tau)$, denote $\widehat{\pa}(k):=\pa_{\widehat{\gr}}(k)$ for any node $k$.
Then if 
\[
n \le \frac{\sigma^2_{\max}}{8M\beta_{\min}^2}d\log \frac{d}{q}\,,
\]
we have
\begin{align*}
    \inf_{F\in \overline{\mathcal{F}}_{d,q}(\beta_{\min},\sigma^2_{\max},M)}\truepr(\wh{\gr} = G(F) \given \tau) & = \inf_{F\in \overline{\mathcal{F}}_{d,q}(\beta_{\min},\sigma^2_{\max},M)}\truepr(\wh{\pa}(k) = \pa(k) \, \forall k \given \tau)\\
    & \le \inf_{F\in \overline{\mathcal{F}}_{d,q}(\beta_{\min},\sigma^2_{\max},M)}\truepr(\wh{\pa}(\tau_d) = \pa(\tau_d) \given \tau) \\
    &\le \inf_{\substack{F\in \overline{\mathcal{F}}_{d,q}(\beta_{\min},\sigma^2_{\max},M) \\ |\pa(\tau_d)|=q, \var(\epsilon_{\tau_d})=\sigma^2_{\max}}}\truepr(\wh{\pa}(\tau_d) = \pa(\tau_d) \given \tau)\\
    &\le \inf_{\substack{F\in \overline{\mathcal{F}}_{d,q}(\beta_{\min},\sigma^2_{\max},M) \\ |\pa(\tau_d)|=q, \var(\epsilon_{\tau_d})=\sigma^2_{\max}}}\truepr(\wh{\pa}(\tau_d) = \pa(\tau_d) \given |\pa(\tau_d)|=q, \tau)\\
    & < \frac{1}{2}
\end{align*}
The first inequality is by relaxing the problem to simply finding the parents of the last node from all preceding nodes. The second inequality is because we can restrict at a sub-ensemble of $\overline{\mathcal{F}}_{d,q}(\beta_{\min},\sigma^2_{\max},M)$ whose last node of ordering has $q$ parents and maximum noise $\var(\epsilon_{\tau_d})=\sigma^2_{\max}$. The third inequality is because knowing the number of parents only makes the problem easier. The final inequality is by noticing the equivalence to sparsity recovery problem and applying Lemma~\ref{lem:given_ordering:lb}.
\end{proof}

\section{Proof of lower bound of GGM (Theorem~\ref{thm:ug:lb})}\label{app:comp}
\begin{proof}
We introduce two useful lemmas from \citet{wang2010information}:
\begin{lemma}[\citet{wang2010information}, Section IV.A]\label{lem:comp:fano}
Consider a restricted ensemble $\widetilde{\mathcal{U}}\subseteq \mathcal{U}$ consisting of $N = |\widetilde{\mathcal{U}}|$ models, and let model index $\theta$ be chosen uniformly at random from $\{1,...,N\}$. Given the observations $X$, the error probability for any estimator $\widehat{U}$
\[
\max_{U\in \mathcal{U}}\truepr(\wh{U}\ne U) \ge \max_{j=1,\ldots,N}\truepr(\wh{U}\ne \widetilde{U}_j)\ge 1- \frac{I(\theta;X) + 1}{\log N}
\]
\end{lemma}
\begin{lemma}[\citet{wang2010information}, Section IV.A]\label{lem:comp:entbound}
Define the averaged covariance matrix
\[
\bar{\Sigma} := \frac{1}{N}\sum_{j=1}^N\Sigma(\widetilde{U}_j)
\]
The mutual information is upper bounded by $I(\theta;X)\le \frac{n}{2}R(\widetilde{\mathcal{U}})$, where
\[
R(\widetilde{\mathcal{U}}) = \log \det\bar{\Sigma} - \frac{1}{N}\sum_{j=1}^N \log \det \Sigma(\widetilde{U}_j)
\]
\end{lemma}
Another lemma for ease of presentation:
\begin{lemma}\label{lem:comp:det}
For a matrix of dimension $p$
\[
A = \begin{pmatrix}
1+a & b & \cdots & b \\
b & 1+a & \cdots & b \\
 &  & \cdots \\
b & b & \cdots & 1+a \\
\end{pmatrix}
\]
with $a,b\to 0$ and $pb\to 0$, the determinant $\log \det A \approx pa$.
\end{lemma}
\begin{proof}
\begin{align*}
    A & = (1+a-b)I_p + b\mathbf{1}_p\mathbf{1}_p\T \\
    & = (1+a-b)\bigg(I_p + \frac{b}{1+a-b}\mathbf{1}_p\mathbf{1}_p\T\bigg) \\
    \det A &= (1+a-b)^p \det \bigg(I_p + \frac{b}{1+a-b}\mathbf{1}_p\mathbf{1}_p\T\bigg) \\
    &= (1+a-b)^p\bigg(1 +\frac{bp}{1+a-b}\bigg)\\
    & = (1+a-b)^{p-1}(1+a+(p-1)b) \\
    \log \det A & = (p-1)\log(1+a-b) + \log (1+a+(p-1)b) \\
    &\approx (p-1)(a-b) + a+(p-1)b = pa.\qedhere
\end{align*}
\end{proof}
Finally, let's consider three ensembles of UGs generated by DAGs. We describe the ensembles by showing how the DAGs generate the UGs.

\paragraph{Ensemble A}
In this first Ensemble, we consider an empty DAG, then add one edge from node $S$ to $T$ with linear coefficient $\beta_{\min}$. Specifically,
\[
\begin{cases}
X_\ell = \beta_{\min}X_S + \epsilon_\ell & \ell = T \\
X_\ell =  \epsilon_\ell  & \ell \ne T 
\end{cases}
\]
Without loss of generality, let $\var(\epsilon_\ell)=1$, general variance $\sigma^2$ would not affect the final results. There are $N=d(d-1)$ possibilities, thus $\log N \asymp \log d$.

It remains to figure out the structure of covariance matrix and find out the corresponding determinants. Without loss of generality, let the first two nodes to be $S,T$, then the covariance matrix of any model (denoted as $j$th) is
\[
\Sigma_j = \begin{pmatrix}
1 & \beta_{\min} & \mathbf{0} \\
\beta_{\min} & 1+\beta_{\min}^2 & \mathbf{0} \\
\mathbf{0} & \mathbf{0} & I_{d-2}
\end{pmatrix}
\]
It is easy to see that $\log\det \Sigma_j =\log(1 + \beta_{\min}^2 - \beta_{\min}\times \beta_{\min}) = 0$ for all models in this subclass. To compute the average $\bar{\Sigma}$, by symmetry, all diagonal and off-diagonal entries are the same respectively. For entries on diagonal, there are two situations: whether it corresponds to node $T$ or not. For off-diagonal entries, there are two situations: corresponds to edge $S-T$ or not. Different situations behave differently Table~\ref{tab:DAGUG:ensemble:B} with total counts $N$:

\begin{table}[h]
\caption{Summary of situations of sub-covariance matrix entry in Ensemble A.}
\label{tab:DAGUG:ensemble:B}
\centering
\begin{tabular}{c|c|c}
           & \multicolumn{2}{c}{diagonal}                         \\ \hline
situations & $T$                               & Otherwise        \\
values     & $\beta_{\min}^2+1$                &   $1$            \\
counts     & $d-1$                             & $(d-1)^2$        \\ \hline
           & \multicolumn{2}{c}{off-diagonal}                     \\ \hline
situations & $T-S$                             & Otherwise        \\
values     & $\beta_{\min}$                    & $0$              \\
counts     & $2$                               & $d(d-1)-2$      
\end{tabular}
\end{table}

Thus we conclude the entries in $\bar{\Sigma}$:
\[
\bar{\Sigma}_{ik} = \begin{cases}
1 + \frac{\beta_{\min}^2}{d} := 1+a & i=k \\
\frac{2\beta}{d(d-1)} :=b &i\ne k
\end{cases}
\]
Using Lemma~\ref{lem:comp:det}, we conclude $\log \det\bar{\Sigma} \asymp \beta^2_{\min}$, and invoking Lemma~\ref{lem:comp:fano} and~\ref{lem:comp:entbound}, we obtain a lower bound as
\[
n \gtrsim \frac{\log d}{\beta_{\min}^2} \,.
\]

\paragraph{Ensemble B}
Here we can adopt the same construction as the first ensemble for DAG in Appendix~\ref{app:ub}, which applies analogously through Lemma~\ref{lem:lb:fano}. Since the joint distribution remains to be the same, we have KL divergence upper bounded by $(M^2-1)d$. 

For number of models inside this class, firstly we know that for a UG with degree bounded by $s$, there are $\Theta(ds\log{d/s})$ many UGs (Lemma 1(b) of \citet{santhanam2012information}). By Lemma~\ref{lem:lohbum}, $U=\mathcal{M}(\gr)$, so $q \le s$, thus the number of UGs would be greater than $\Theta(dq\log(d/q))$, which leads to the same lower bound:
\begin{align*}
n &\gtrsim \frac{q\log(d/q)}{M^2-1}\,.\qedhere
\end{align*}
\end{proof}

\section{Proof of Lemma~\ref{lem:Delta}}\label{app:Delta}
\begin{proof}
Immediate from the law of total variance:
\begin{align*}
    \Delta & \equiv \min_k\min_{\substack{\anc\subseteq\nd(k) \\  \pa(k) \setminus \anc \ne \emptyset \\ \anc\subseteq \nd(\pa(k) \setminus \anc)}}\E_\anc\var(X_k\given \anc) - \sigma^2 \\
    & = \min_k\min_{\substack{\anc\subseteq\nd(k) \\  \pa(k) \setminus \anc \ne \emptyset \\ \anc\subseteq \nd(\pa(k) \setminus \anc)}} \E_\anc \E_{\pa(k)\setminus \anc}\var(X_k\given \pa(k))\\
    & \ \ \ \ + \E_\anc \var_{\pa(k)\setminus \anc} \E(X_k\given \pa(k)) - \sigma^2 \\
    & = \min_k\min_{\substack{\anc\subseteq\nd(k) \\  \pa(k) \setminus \anc \ne \emptyset \\ \anc\subseteq \nd(\pa(k) \setminus \anc)}} \E_\anc\var_{\pa(k)\setminus \anc} \bigg[ \beta_{\pa(k)\setminus \anc}\T X_{\pa(k)\setminus \anc} \given \anc\bigg] \\
    & = \beta_{\min}^2 \sigma^2.
    \qedhere
\end{align*}
\end{proof}

\end{document}